\documentclass[9pt, english, a4paper]{article}
\usepackage[utf8]{inputenc}
\usepackage[T1]{fontenc}
\usepackage{babel, graphicx, textcomp, varioref, amsfonts}
\usepackage{amsthm}
\usepackage{amssymb}
\usepackage{amsmath}
\usepackage{mathrsfs}
\usepackage{color}
\usepackage{mathtools}

\usepackage[margin=1.5in]{geometry}

\newtheorem{theorem}{Theorem}[section]
\newtheorem{definition}[theorem]{Definition}
\newtheorem{lemma}[theorem]{Lemma}
\newtheorem{proposition}[theorem]{Proposition}

\newtheorem{remark}[theorem]{Remark}

\newcommand{\R}{\mathbb{R}}

\newcommand{\B}{\mathbb{B}}

\newcommand{\Z}{\mathbf{Z}}
\newcommand{\ZZ}{\mathbb{Z}}
\newcommand{\W}{\mathbf{W}}
\newcommand{\WW}{\mathbb{W}}

\title{Weak solutions of rough path SDE's via Girsanov}

\author{Torstein Nilssen \thanks{Institute of Mathematics, Technical University of Berlin, Germany, Financial support by the DFG via Research Unit FOR 2402 is gratefully acknowledged.} }

\date{}

\begin{document}

\maketitle

\begin{abstract}
We consider a differential equation driven by a Brownian motion as well as a rough path. We prove a Girsanov-type result for this equation to construct a weak solution in the probabilistic sense. 

\bigskip

MSC Classification Numbers: 60H05, 60H10.


\end{abstract}

\section{Introduction}

In this note we show a Girsanov type result to prove existence and uniquness of weak solutions 
%a Girsanov type theorem 
of rough stochastic ordinary differential equations on the form
\begin{equation} \label{eq:SmoothSDE}
dX_t = u(t,X_t)dt + \sqrt{ 2 \nu}dB_t + \beta_j(X_t) \dot{Z}^j_t dt , \qquad X_0 = x \in \R^d
\end{equation}
where $u: [0,T] \times \R^d \rightarrow \R^d$ is a bounded and measurable vector field, $B$ is a Brownian motion, $\beta_j : \R^d \rightarrow \R^d$ are smooth vector fields and $Z$ can be lifted to a (deterministic) rough path. Above and below we use the convention of summation over repeated indices.

We motivate the study of this equation by the well known relation (at least in the case when $u$ and $Z$ are smooth) it has to the Kolmogorov equation
\begin{equation} \label{eq:Kolmogorov}
\partial_t \xi = \nu \Delta \xi + (u \nabla )\xi +    \beta_j \nabla \xi \dot{Z}_t^j 
\end{equation}
via the Feynman-Kac formula
\begin{equation} \label{eq:FK}
\xi(t,x) = E[ \xi_0(X_t^x)] ,
\end{equation}
which is classical for a smooth path $Z$. Moreover, this relation is shown to hold also when $Z$ is a rough path, as shown in \cite{DFS} and \cite{FNS}. 
See also the recent work \cite{HH} where an intrinsic notion of solution, as introduced in \cite{BG}, was used to show well-posedness of \eqref{eq:Kolmogorov} under optimal conditions on the coefficients when written in divergence form.

Understanding how to add the term $\beta_j \nabla \xi \dot{Z}_t^j$ could be motivated by the regularization by noise problem. Much attention has been given the degenerate version of \eqref{eq:Kolmogorov}, i.e. when $\nu = 0$, and $u$ is some irregular vector field. In particular \cite{FedrizziFlandoli} and \cite{MNP} proved well-posedness of the equation when $Z$ is a Brownian motion, $\beta_j(x) = e_j$ (i.e. the $j$th basis vector) and $u$ is allowed to be discontinuous. 
In these papers, the stochastic product was in the sense of Stratonovich, which will be the same as in the present paper by choosing a geometric rough path. 
Moreover, \cite{Catellier} and \cite{Nilssen} study a similar problem when $Z$ is a fractional Brownian motion.

To the best of the authors knowledge there has been no study of the corresponding diffusive equations, which in any case is well posed without noise in the linear case. However, by replacing $\xi$ by $u$ in \eqref{eq:FK} and solving \eqref{eq:SmoothSDE} and \eqref{eq:FK} as a \emph{system}, one can also find stochastic representations of nonlinear equations. Coupling this system also with the inverse of \eqref{eq:SmoothSDE}  and the Biot-Savart law to \eqref{eq:FK} the papers \cite{BarbaraBusnello} and \cite{BusnelloFlandoliRomito} show a stochastic representation of the Navier-Stokes equation, when $Z = 0$. In these papers, an integral part of the technique is the Girsanov transform.

\section{Notation and preliminary results}

\subsection{H\"{o}lder spaces and rough paths}

For a Banach space $E$, an integer $J \geq 1$, $\alpha >0$ and an interval $I \subset [0,T]$ denote by $C^{\alpha}_2(I ; E)$ the space of all continuous mappings $g : \{ (s,t) \in I^2 :  s \leq t  \} \rightarrow E$  such that $|g_{st}| \lesssim |t-s|^{\alpha}$ and denote by $[g]_{\alpha;I}$ the smallest constant satisfying this inequality. When $I = [0,T]$ we will simply write $[g]_{\alpha}$. 
For a path $f$ we will abuse notation and write $[f]_{\alpha}$ when we mean $[\delta f]_{\alpha}$ where $\delta f_{st} = f_t - f_s$.
We shall say that $\Z = (Z,\ZZ) \in C^{\alpha}_2([0,T] ; \R^J) \times C^{2 \alpha}_2([0,T] ; \R^{J \times J})$ with $\alpha \in (\frac13, \frac12)$ is a \emph{rough path} provided \emph{Chen's relation} holds, i.e.
\begin{equation} \label{chen}
\delta Z_{s \theta t} = 0 \quad \textrm{ and } \quad 
\delta \ZZ_{s \theta t}^{i,j} =  Z_{s \theta}^i Z_{\theta t}^j,
\end{equation}
where for a 2-index map $g$ we define $\delta g_{s \theta t} := g_{s t} - g_{\theta t} - g_{s \theta}$ for $s < \theta < t$. The left equality in \eqref{chen} implies that $Z$ is a path, and for simplicity we will assume that this path starts at $0$.
We denote by $\mathscr{C}^{\alpha}([0,T];\R^J)$ the set of all rough paths with topology induced by the metric of $C^{\alpha}_2([0,T] ; \R^J) \times C^{2 \alpha}_2([0,T] ; \R^{J \times J})$. A rough path $\Z$ is said to be \emph{geometric} provided there exists a sequence of smooth paths $Z^n :[0,T] \rightarrow \R^J$ such that $\Z^n = (Z^n, \ZZ^n) \rightarrow \Z$
where we have defined $\ZZ_{st}^{n,i,j}  = \int_s^t \delta Z_{ s r}^{n,i} \dot{Z}^{n,j}_r dr $.

The \emph{sewing lemma} (see \cite[Lemma 4.2]{FH14}) tells us that given a 2-index map $g$ such that $|\delta g_{s \theta t} | \leq |t-s|^{\zeta}$ for some $\zeta >1$, there exists a unique pair $( I(g), I^{\natural} (g))$ such that 
$$
\delta I(g)_{st}   = g_{st} + I^{\natural}(g)_{st}
$$
where $I(g)$ is a path and $|I^{\natural}(g)_{st}| \leq C_{\zeta} |t-s|^{\zeta}$ where $C_{\zeta}$ depends only on $\zeta$. The mapping $I$ is linear in $g$, and so easily extends to the case $|\delta g_{s \theta t} | \lesssim |t-s|^{\zeta}$.

\subsection{Girsanov's theorem and weak solutions} \label{section:Girsanov}

We recall the Girsanov change of measure. Let $(\Omega, \mathcal{F},P)$ be a probability space with a $d$-dimensional Brownian motion $(B_t, \mathcal{F}_t)_{t \in [0,T]}$. If $v \in L^{\infty}([0,T] \times \Omega; \R^d)$ is adapted to $\mathcal{F}_t$, the 
\emph{Girsanov theorem} (see \cite[Theorem 3.5.1]{KaratzasShreve}) states that if we define
$$
dQ = \exp \left( - \int_0^T v_s^j dB^j_s - \frac12 \int_0^T |v_s|^2 ds \right)  dP ,
$$
then $Y_t := B_t + \int_0^t v_s ds$ is a Brownian motion on $(\Omega, \mathcal{F},Q)$.

This leads to the construction of solutions to the equation
\begin{equation} \label{eq:BMplusDrift}
dX_t = b(t,X_t) dt + B_t,
\end{equation}
as follows. 
Let $(\Omega, \mathcal{F},P)$ be a probability space with a Brownian motion $( \tilde{B}_t, \mathcal{F}_t)_{t \in [0,T]}$ and assume $b \in L^{\infty}([0,T] \times \R^d ; \R^d)$. 
Define the probability measure 
$$
dQ := \exp \left(  \int_0^T b_j(s,x + \tilde{B}_s) d\tilde{B}^j_s - \frac12 \int_0^T |b(s,x + \tilde{B}_s)|^2 ds \right) dP.
$$
Then $Q$ is such that $B_t :=  \tilde{B}_t - \int_0^t b(s,x + \tilde{B}_s) ds$ is a Brownian motion on $(\Omega, \mathcal{F},Q)$. 
If we define $X_t := x + \tilde{B}_t$ we see that $X$ satisfies \eqref{eq:BMplusDrift}.

We note that the filtration generated by the Brownian motion, $(\mathcal{F}^B_t)_{t \in [0,T]}$, is by construction contained in $(\mathcal{F}^X_t)_{t \in[0,T]} =  (\mathcal{F}^{\tilde{B}}_t)_{t \in[0,T]}$ and we call such a solution a \emph{weak solution}.
A priori it is not clear that if these filtrations actually coincide, but this has been shown to be true under very general assumptions on the drift, see e.g. \cite{Zvonkin} and \cite{Veretennikov}. We conjecture that similar results will hold for the equation \eqref{MainEq}.

On the other hand, if we include a diffusion coefficient, one can easily find examples for which we have $\mathcal{F}^B \neq \mathcal{F}^{\tilde{B}}$, e.g. the Tanaka equation, see \cite[Example 5.3.5]{KaratzasShreve}.

%\bigskip

\section{Formal computations}

For simplicity we assume $\nu = \frac12$, and for the rest of the paper we want to study the equation
\begin{equation} \label{MainEq}
dX_t = u(t,X_t) dt + \beta_j(X_t) d\Z_t^j + dB_t, \qquad X_0 = x \in \R^d
\end{equation}
where $\Z$ is a rough path, $u$ is a bounded function and $B$ is a Brownian motion w.r.t. some probability space.

\subsection{A singular measure approach} \label{section:SingularMeasureFormal}

We first try to copy the approach in Section \ref{section:Girsanov} directly to construct a weak solution.

Assume for simplicity $J=d=1, x=0$, $\beta = 1$ and $Z$ is smooth, so that we are looking to construct a weak solution to
$$
dX_t = u(t,X_t)dt + dB_t +  dZ_t , \qquad X_0 = 0 
$$

Let $\tilde{B}$ be a Brownian motion on $(\Omega, \mathcal{F}, P)$. If we define 
$$
B_t := \tilde{B}_t - \int_0^t u(s,\tilde{B}_s) ds - Z_t,
$$
the next step is to change the measure such that $B$ is a Brownian motion under some new probability measure. %If we can do this, we are done.
The Girsanov change of measure reads
$$
dQ = \exp \left( \int_0^T (u(t, \tilde{B}_t ) + \dot{Z}_t) d \tilde{B}_t  - \int_0^T \left|  u(t, \tilde{B}_t ) + \dot{Z}_t   \right|^2  dt \right) dP.
$$

To illustrate the idea, assume that $u = 0$, and define
$$
M_T(Z) := \exp \left( \int_0^T  \dot{Z}_t d \tilde{B}_t - \frac12 \int_0^T |\dot{Z}_t|^2  dt \right)  
$$

Since $Z$ is deterministic, we get 
$$
E[M_T(Z)^p] = \exp \left( \frac{p}{2} (p - 1) \int_0^T |\dot{Z}_t|^2 dt \right) ,
$$
so that if we approximate a truly rough path $Z$, in the sense that $\int_0^T |\dot{Z}_s|^2 ds = \infty$, by a sequence of smooth paths $Z^n$ we get 
$$
\lim_{n \rightarrow \infty} E[M_T(Z^n)^p] = \left\{ \begin{array}{ll}
1 & \textrm{ if } p=1 \\
\infty & \textrm{ if } p > 1 \\
\end{array} \right. .
$$

This means that there is no hope in taking a strong limit of $M(Z^n)$ or even a weak limit in $L^p(\Omega,\mathcal{F},P)$ for $p >1$. Instead, let us try to take a weak limit of the induced measure $dQ^n = M_T(Z^n)dP$. 
Without loss of generality we can assume $\Omega = C([0,T];\R^d)$, $\mathcal{F}$ is the Borel $\sigma$-algebra and $P$ is the Wiener measure. 
Take an open set $A \subset C([0,T];\R^d)$. Then 
$$
  Q^n (A) = P( \omega + Z^n \in A) \rightarrow P(\omega + Z \in A) =: Q(A)
$$
However, by assumption $Z$ does not belong to the Cameron-Martin space of $P$
%$$
%\mathcal{H} := \left\{ \int_0^{\cdot} h_s ds : h \in L^2([0,T]) \right\}
%$$
and so $Q$ can not be absolutely continuous w.r.t. to the Wiener-measure. 

Even worse, even if $B^n(\omega) = \omega + Z^n$ is a Brownian motion w.r.t. $Q^n$ for every $n$ and we have strong (respectively weak) convergence to $B(\omega) = \omega + Z$ (respectively $Q$) the limiting process is in general not a Brownian motion. In fact, assume $Z$ is a path which is truly rougher than the Brownian paths, e.g. the sample path of a fractional Brownian motion with Hurst parameter $H< 1/2$. Then, if $B$ was a Brownian motion w.r.t. $Q$, its sample paths could be chosen to be H\"{o}lder continuous with an exponent strictly bigger than $H$, giving a contradiction.

\subsection{An equivalent measure approach} \label{section:formalComputation}

To construct a weak solution with an equivalent measure we do the following. 

Let $\tilde{B}$ be a Brownian motion on $(\Omega, \mathcal{F}, P)$. Solve the equation 
\begin{equation} \label{eq:WeakSolutionEquation}
dX_t = d \tilde{B}_t + \beta_j(X_t) d\Z^j_t 
\end{equation}
%Equation has a priori no meaning, has to be interpreted as 
%$$
%dX_t = (1, \beta)(X_t) d \left(
%\begin{array}{l}
%\tilde{B}_t \\
%Z_t \\
%\end{array} \right)
%$$
%so need to construct the rough path above $(\tilde{B}, Z)$ which can be done canonically using the Wiener integral. 
%
and define
$$
B_t = \tilde{B}_t - \int_0^t u(s,X_s) ds.
$$
%where $X$ is as above. 
%It is indeed measurable w.r.t. the filtration generated by $\tilde{B}$ since it can be obtained as a limit of measurable processes constructed by approximating the rough path $Z$.
%
If $X$ is adapted to the filtration generated by $\tilde{B}$ we may construct the It\^{o} integral $\int_0^T u(s,X_s) d \tilde{B}_s$. If $u$ is bounded %(or satisfy some type of Novikov condition - to be explored) 
then it is clear that 
$$
dQ :=  \exp \left( \int_0^T u_j(t, X_t )  d \tilde{B}^j_t  - \frac12 \int_0^T |u(t,X_t)|^2 dt \right) dP.
$$
is such that $B$ is a Brownian motion w.r.t $(\Omega, \mathcal{F}, Q)$. Moreover, by the definition of $X$ we have
\begin{equation} \label{eq:FormalDerivatives}
d B_t = d\tilde{B}_t - u(t,X_t) dt = dX_t - \beta_j(X_t) d\Z^j_t - u(t,X_t) dt 
\end{equation}
which shows that $X$ is a weak solution.

\bigskip

The expert reader will notice that solving the equation \eqref{eq:WeakSolutionEquation} and the computation \eqref{eq:FormalDerivatives} needs extra care to be done rigorously in the rough path setting. 

\section{Main results}

The rest of the paper is devoted to making the computations in Section \ref{section:formalComputation} rigorous. We start by introducing the correct notion of a solution to \eqref{MainEq}.

\subsection{Rough solutions}
Let us first consider the equation \eqref{MainEq} without the Brownian motion, i.e. 
\begin{equation} \label{eq:NoBM}
dX_t  = u(X_t) dt + \beta_j(X_t) d\Z^j_t ,
\end{equation}
where $\Z$ is a rough path. 
One way of doing this is to transform the rough path $\Z$ into a rough path containing the drift term,
\begin{equation} \label{eq:RPplusTime}
\hat{Z}_t = \left( \begin{array}{c}
t \\
Z_t \\
\end{array} \right) 
 \quad  \textrm{ and }\quad
\hat{\ZZ}_{st} = \left( \begin{array}{cc}
\frac12 (t-s)^2 & \int_s^t \delta Z_{sr} dr \\
\int_s^t (r-s) dZ_r & \ZZ_{st}  \\
\end{array} \right)
\end{equation}
where all the above terms are well defined since $t \mapsto t$ is smooth, and then define $V_0(x) = u(x)$, $V_j(x) = \beta_j(x)$ for $j=1, \dots , J$. One could then solve the equation
\begin{equation*} %\label{eq:NoBMNoDrift}
dX_t  = V_j(X_t) d\hat{\Z}^j_t 
\end{equation*}
in the rough path sense.

This would however require higher regularity of the coefficient $u$ than what is the aim of this paper. %Moreover, 
Exploiting the original structure of the equation \eqref{eq:NoBM} is also done in \cite{FO09} where the authors show the well posedness of the equation for $u \in C^1_b(\R^d)$.

We introduce the notion of a pathwise solution as first defined by Davie in \cite{bib:davie}. The only difference is that we allow for a drift term. 
\begin{definition} \label{def:RPsolution}
A path $X :[0,T] \rightarrow \R^d$ is said to be a solution to \eqref{eq:NoBM} provided 
\begin{equation} \label{eq:remainder}
X_{st}^{\natural} :=  \delta X_{st} - \int_s^t u(r, X_r) dr  - \beta_j(X_s) \delta Z_{st}^j - \nabla  \beta_j(X_s) \beta_i(X_s) \ZZ_{st}^{i,j}
\end{equation}
is a \emph{remainder}, i.e. $|X_{st}^{\natural}| \lesssim |t-s|^{\zeta}$ for some $\zeta > 1$. 

\end{definition} 

\begin{remark}
The terms $\beta_j(X_s) \delta Z_{st}^j + \nabla \beta_j(X_s) \beta_i(X_s) \ZZ_{st}^{i,j} +X_{st}^{\natural}$ represent the rough path integral $\int_s^t \beta_j(X_r) d\Z_r$. In fact, from the sewing lemma, there exists a path $I(g)$ representing the integral of the local expansion
$$
g_{st} := \beta_j(X_s) \delta Z_{st}^j + \nabla\beta_j(X_s) \beta_i(X_s) \ZZ_{st}^{i,j} ,
$$
where one can use \eqref{chen} and the assumptions on $\beta_j$ to check that $|\delta g_{s \theta t} | \lesssim |t-s|^{3 \alpha}$ (recall that $3\alpha >1$ by assumption).
From the uniqueness in the sewing lemma, if $X_{st}^{\natural}$ is a remainder, it is clear that 
$$
\delta X_{st} - \int_s^t u(r,X_r) dr  =  \delta I(g)_{st} =:  \int_s^t \beta_j(X_r) d\Z_r . 
$$
\end{remark}

The main result we shall need on rough path differential equations is the following result when $u = 0$, see \cite{FH14} for a proof and a nice introduction to rough path theory.

\begin{theorem} \label{thm:RDEwellPosed}
Assume $\beta_j \in C^3_b(\R^d)$ and $\Z$ is a rough path. Then there exists a unique solution to 
$$
dX_t = \beta_j(X_t)d\Z^j_t, \qquad X_0 = x \in \R^d .
$$
Moreover, the mapping $\Z \mapsto \delta X$ is locally Lipschitz from $\mathscr{C}^{\alpha}([0,T];\R^J)$ to $C^{\bar{\alpha}}_2([0,T];\R^d)$ for all $\bar{\alpha} < \alpha$.
\end{theorem}

\bigskip

As explained in the introduction, we want to consider the equation
$$
dX_t = u(t,X_t) dt + \beta_j(X_t) d\Z^j_t + dB_t.
$$ 
To understand what a notion of a solution to this equation should be, assume for simplicity that $u=0$. Motivated by Definition \ref{def:RPsolution}, one could be tempted to say that a solution is a function $X$ such that 
$$
X_{st}^{\natural} :=  \delta X_{st} - \delta B_{st}  - \beta_j(X_s) \delta Z_{st}^j - \nabla \beta_j(X_s) \beta_i(X_s) \ZZ_{st}^{i,j} 
$$
is a remainder. 
This definition, however, does not contain "area" between $Z$ and $B$ and is thus not suitable. More specifically, one can check that the local expansion
$$
g_{st} :=   \delta B_{st}  + \beta_j(X_s) \delta Z_{st}^j + \nabla \beta_j(X_s) \beta_i(X_s) \ZZ_{st}^{i,j}
$$
does in general \emph{not} satisfy $| \delta g_{s \theta t} | \lesssim |t-s|^{\zeta}$ for any $\zeta > 1$, so that the sewing lemma does not apply. This is not in conflict with Definition \ref{def:RPsolution} since there the drift term is of bounded variation, which is related to the fact that the area in \eqref{eq:RPplusTime} is canonically defined.

Moreover, with a view towards adding a diffusion-coefficient to study the equation 
$$
dX_t = u(t,X_t) dt + \beta_j(X_t) d\Z^j_t + \sigma_k(X_t) dB^k_t,
$$
one really needs to understand $(B,Z)$ as a rough path. This is done in the next section. 

\subsection{Joint lift} \label{section:JointLift}

Let $\Z = (Z, \ZZ)$ be a geometric rough path. Given a Brownian motion $B$ on some probability space $(\Omega, \mathcal{F}, P)$, we may construct $ \int_s^t \delta Z_{sr} dB_r$ as the Wiener-It\^{o} integral, in particular an element of $L^2(\Omega, \mathcal{F},P; \R^{J \times d})$ such that 
$$
E \left[ \left| \int_s^t \delta Z_{sr} dB_r \right|^2 \right] =    \int_s^t  |\delta Z_{sr}|^2 dr .
$$
The above right hand side can be bounded by $[Z]_{\alpha} \int_s^t |r-s|^{2 \alpha} dr \lesssim  |t-s|^{2 \alpha + 1}$.

Moreover, since $Z$ is deterministic, $\int_s^t \delta Z_{sr} dB_r$ is a Gaussian random variable, and by equivalence of moments and the Kolmogorov theorem for 2-index maps (see \cite[Theorem 3.1]{FH14}) we get that there exists a set $N \in \mathcal{F}$ such that $P(N) = 0$ and for all $\omega \in N^c$ we have $ \left| \int_s^t \delta Z_{sr} dB_r (\omega) \right| \leq K_{\alpha}(\omega) |t-s|^{2 \alpha}$ for some $K_{\alpha} \in L^p(\Omega, \mathcal{F},P)$. 
Define now the rough path $\tilde{\Z} = \tilde{\Z}(B)$ by

$$
\tilde{Z}_t = \left( \begin{array}{c}
B_t \\
Z_t \\
\end{array} \right) 
 \quad  \textrm{ and }\quad
\tilde{\ZZ}_{st} = \left( \begin{array}{cc}
\mathbb{B}_{st} & \int_s^t \delta Z_{sr} dB_r \\
\int_s^t \delta B_{sr} dZ_r & \ZZ_{st}  \\
\end{array} \right)
$$
on the set $\Omega_0 := \{  \delta B_{\cdot} \in C^{\alpha}_2([0,T]; \R^d) \} \cap \{ \mathbb{B} \in C_2^{2 \alpha}([0,T]; \R^{d \times d}) \} \cap N^c$, which has full $P$-measure. Above we have defined $\mathbb{B}_{st} = \int_s^t \delta B_{sr} dB_r$ and 
$$
\int_s^t \delta B_{sr} dZ_r :=  \delta B_{st} \delta Z_{st}  - \int_s^t \delta Z_{sr} dB_r 
$$
so that we have $| \int_s^t \delta B_{sr} dZ_r | \leq |t-s|^{2 \alpha}([B]_{\alpha} [Z]_{\alpha} + K_{\alpha} [Z]_{\alpha}) $ on $\Omega_0$. 

It is easy to check that $\tilde{\Z}$ is a rough path. Moreover, it is shown in \cite{DOR} that the mapping
\begin{align*}
\mathscr{C}^{\alpha}([0,T];\R^J) & \rightarrow L^p(\Omega, \mathcal{F}, P; \mathscr{C}^{\bar{\alpha}}([0,T];\R^{d+J}) )  \\
\Z & \mapsto \tilde{\Z} 
\end{align*}
is locally Lipschitz for any $\bar{\alpha} < \alpha$.

\begin{remark}
From \cite[Theorem 3.1]{FH14} it is clear that one can choose $\int_{\cdot}^{\cdot} \delta Z_{\cdot r} dB_r \in C^{\gamma}_2([0,T]; \R^{J+d})$ for $\gamma \in (0, 2 \alpha +1)$. 
\end{remark}

\begin{remark}
We did not specify which type of integration we used for the definition of $\mathbb{B}$ since we are considering constant diffusion vector fields. It can be checked (and is in fact spelled out in \eqref{eq:expansion}) that the solution is independent of this choice.  
\end{remark}

\subsection{Pathwise weak solutions}

Using the previous section we are able to define the notion of a weak solution of \eqref{MainEq}.

\begin{definition}
Given $u$ and the rough path $\Z$, we say that a probability space $(\Omega, \mathcal{F}, P)$ supporting a Brownian motion $B$ and a stochastic process $X :[0,T] \times \Omega \rightarrow \R^d$ is a \emph{weak solution} to \eqref{MainEq} provided there exists a set of full $P$-measure where 
\begin{equation} \label{eq:RoughSDE}
dX_t = u(t,X_t) dt + V_j(X_t) d \tilde{\Z}(B)^j_t , \quad X_0 = x \in \R^d
\end{equation}
in the sense of Definition \ref{def:RPsolution}. Above we have defined the joint lift $\tilde{\Z}(B)$ as in Section \ref{section:JointLift} and we have defined $V_j = e_j$ for $j=1, \dots, d$ and $V_j = \beta_j$ for $j = d+1, \dots, d + J$. 

\end{definition}

\begin{remark}
Notice that the filtration in the above definition is in general not generated by the Brownian motion $B$.
\end{remark}

With a proper definition of a solution in place, we go on to prove existence and uniqueness in law.

\subsection{Existence}

\begin{proposition} \label{proposition}

Let $(\Omega, \mathcal{F}, P)$ be a probability space with a Brownian motion $(B_t, \mathcal{F}_t)_{t \in [0,T]}$. Suppose $\Z$ is a geometric rough path and $\beta_j \in C^3_b(\R^d)$ for all $j$. Then there exists a set $\Omega_0$ of full $P$ measure such that for all $\omega \in \Omega_0$ there exists a unique solution to  
\begin{equation} \label{eq:PureRoughEq}
dX_t(\omega) = V_j(X_t(\omega)) d \tilde{\Z}(B)_t^j(\omega), \quad X_0 = x \in \R^d
\end{equation}
where $V$ and $\tilde{\Z}(B)$ are as defined in the previous section. Moreover, the resulting stochastic process $X$ has a modification which is $\mathcal{F}_t$-adapted.
\end{proposition}

\begin{proof}

Let $\Omega_0$ be as in the definition of $\tilde{\Z}(B)(\omega)$, i.e. such that $\tilde{\Z}(B)(\omega)$ is a rough path. Existence and uniqueness of a rough path solution to \eqref{eq:PureRoughEq} is then classical under the assumption $V_j \in C_b^3(\R^d)$. This is clearly true when $\beta_j \in C^3_b(\R^d)$.

To prove that the solution is adapted, we shall use a more general result from the following lemma which is also proved in \cite[Theorem 8]{DOR}.

\begin{lemma} \label{lemma:RPcont}
For any $p \in [1,\infty)$ the mappings 
\begin{equation} \label{diagram}
\begin{array}{ccccc}
\mathscr{C}^{\alpha}([0,T]; \R^J) &  \rightarrow  & L^p(\Omega, \mathcal{F},P; \mathscr{C}^{\bar{\alpha}}([0,T]; \R^{d+J}))  & \rightarrow &  L^p(\Omega, \mathcal{F},P; C_2^{\bar{\bar{\alpha}}}([0,T]; \R^d)) \\
 \Z & \mapsto  & \tilde{\Z}(B) &  \mapsto & \delta X^{\Z} ,
\end{array}
\end{equation}
are locally Lipschitz for $\bar{\bar{\alpha}} < \bar{\alpha} < \alpha$. 
\end{lemma}

\begin{proof}
For a bounded path $Z$ we define the stopping times 
$$
\tau_n^{Z} = \inf \left\{ t > 0 :   [ B]_{ \alpha ; [0,t]}  \geq n , [ \B ]_{ 2 \alpha ; [0,t]} \geq n \textrm{ or } [ \int \delta Z_{\cdot r} dB_r ]_{2 \alpha; [0,t]} \geq n, \right\}  .%, \quad \tau_n^B = \inf \left\{ t > 0 : \left| B_t \right| \geq n \right\}
$$
Restricting to the time interval $\left[0,\tau_n^{Z}\right]$ we have $[\tilde{Z}(B)]_{\alpha; [0,\tau_n^{Z}] } + [\tilde{\ZZ}(B)]_{2 \alpha ;[0,\tau_n^{Z}]} \lesssim n
$ and $P$-a.s. we have $ \lim_{n \rightarrow \infty} \tau_n^Z = T$. 

Let $R> 0$ and $\Z, \W \in \mathscr{C}^{\alpha}([0,T] ; \R^J)$ belong to the ball of radius $R$.
From Theorem \ref{thm:RDEwellPosed} we get the restriction to the time interval $[0, \tau_n^Z \wedge \tau_n^W]$
$$
[ X^{\Z} -  X^{\W} ]_{ \bar{\bar{\alpha}}; [0, \tau_n^Z \wedge \tau_n^W]} \leq C_R \left( [\tilde{Z}(B) - \tilde{W}(B)]_{ \bar{\alpha};[0, \tau_n^Z \wedge \tau_n^W]} + [\tilde{\ZZ}(B) - \tilde{\WW}(B)]_{ 2 \bar{\alpha} ; [0, \tau_n^Z \wedge \tau_n^W]} \right)
$$
where we have chosen $n \gtrsim R$, and we remark that $C_R$ is deterministic. Integrated to the $p$th power and letting $n \rightarrow \infty$ we get by monotone convergence 
\begin{align*}
\left\| [ X^{\Z} -  X^{\W} ]_{ \bar{\bar{\alpha}}} \right\|_{L^p(\Omega,\mathcal{F},P)} &  \leq C_R \left\| [\tilde{Z}(B) - \tilde{W}(B)]_{ \bar{\alpha}} + [\tilde{\ZZ}(B) - \tilde{\WW}(B)]_{ 2 \bar{\alpha} } \right\|_{L^p(\Omega,\mathcal{F},P)} \\
&  \leq C_{R,p}  ( [Z - W]_{ \alpha } + [\ZZ - \WW]_{ 2 \alpha } ) 
\end{align*}
where we have used \cite[Theorem 3]{DOR} in the last step. 
\end{proof}

%\bigskip

To finalize the proof of Proposition \ref{proposition} we need to show that there exists a modification of $X$ which is $\mathcal{F}_t$-adapted. Using Lemma \ref{lemma:RPcont} this is straightforward. Indeed, let $Z^n$ be a sequence of smooth paths such that $\Z^n$ converges to $\Z$ in the rough path topology. For every $n$, there exists a unique stochastic process such that
$$
dX_t^n = dB_t + \beta(X_t^n) \dot{Z}_t^n dt,
$$
and $X_t^n$ is $\mathcal{F}_t$-measurable for every $t$ and $n$. 
The rough path continuity gives that there exists a subsequence (still denoted $X^n$) such that $X_t^n \rightarrow X_t$ for every $t$ on a set of full measure. The result follows since measurability is closed under limits. 
\end{proof}

Note that since the construction $\tilde{\Z}(B)(\omega) \rightarrow X$ is deterministic, the law of $X$ is completely determined by the law of the Brownian motion.

\bigskip

We now proceed to prove existence of a weak solution.

\begin{theorem} \label{thm:RSDEexistence}
Assume $u$ is a bounded function, $\beta_j \in C^3_b(\R^d)$ for all $j$ and $\Z$ is a geometric rough path. Then there exists a weak solution to \eqref{MainEq}.
\end{theorem}

\begin{proof}

Let $(\Omega, \mathcal{F},P)$ be a probability space with Brownian motion $(\tilde{B}_t, \mathcal{F}_t)_{t \in [0,T]}$. Denote by $X$ the $\mathcal{F}_t$-adapted stochastic process solving \eqref{eq:PureRoughEq}, i.e.
\begin{align}
\delta X_{st} & = V_j(X_s) \delta \tilde{Z}(\tilde{B})_{st}^j + \nabla V_j(X_s) V_i(X_s) \delta \tilde{\ZZ}(\tilde{B})_{st}^{i,j} + X_{st}^{\natural} \notag \\
 & = \delta \tilde{B}_{st} + \beta_j(X_s) \delta Z_{st}^j  + \nabla \beta_j(X_s) \beta_i(X_s) \ZZ_{st}^{i,j} + \nabla \beta_j(X_s) e_i \int_s^t \delta \tilde{B}_{sr}^i dZ_r^j + X_{st}^{\natural} \label{eq:expansion}
\end{align}
where $X^{\natural}$ is remainder. 

Define now 
\begin{equation} \label{eq:NewBM}
B_t := - \int_0^t u(s,X_s) ds + \tilde{B}_t
\end{equation}
and the measure
$$
dQ = \exp \left( \int_0^T u_j(s,X_s) d \tilde{B}^j_s - \frac12 \int_0^T |u(s,X_s)|^2 ds \right) dP,
$$
so that $B$ is Brownian motion w.r.t. $(\Omega, \mathcal{F},Q)$. Now, there exists a set $\Omega_1$ of full $Q$-measure such that $\tilde{\Z}(B)$ is a rough path. Since $Q$ and $P$ are equivalent, $\Omega_1$ also has full $P$-measure, and we have
$$
\int_s^t \delta Z_{sr} dB_r = \int_s^t \delta Z_{sr} d\tilde{B}_r - \int_s^t \delta Z_{sr} u(r,X_r)dr .
$$
By the boundedness of $u$ we have
$$
\left| \int_s^t \delta Z_{sr} dB_r -  \int_s^t \delta Z_{sr} d\tilde{B}_r \right| \leq [Z]_{\alpha} \|u\|_{\infty} |t-s|^{1 + \alpha}
$$
for every $\omega$ such that both stochastic integrals are well defined. 

Consequently, plugging \eqref{eq:NewBM} into \eqref{eq:expansion} we get
\begin{align}
\delta X_{st} & =  \int_s^t u(r,X_r) dr + \delta B_{st} + \beta_j(X_s) \delta Z_{st}^j  + \nabla \beta_j(X_s) \beta_i(X_s) \ZZ_{st}^{i,j} \\
 & \qquad + \nabla \beta_j(X_s) e_i \int_s^t \delta B_{sr}^i dZ_r^j + \tilde{X}_{st}^{\natural}  %\label{eq:expansion}
\end{align}
where we have defined 
\begin{align*}
\tilde{X}_{st}^{\natural} = X_{st}^{\natural} + \nabla \beta_j(X_s) e_i \left(  \int_s^t \delta \tilde{B}_{sr}^i dZ_r^j -  \int_s^t \delta B_{sr}^i dZ_r^j \right).
\end{align*}
Now, by definition 
\begin{align*}
\int_s^t \delta \tilde{B}_{sr}^i dZ_r^j   -  \int_s^t \delta B_{sr}^i dZ_r^j & =  (\delta \tilde{B}_{st}^i - \delta \tilde{B}_{st}^i) \delta Z_{st}^j +  \int_s^t \delta Z^j_{sr} dB^i_r -  \int_s^t \delta Z^j_{sr} d\tilde{B}^i_r \\
& = - \int_s^t u(r,X_r)^i dr  \delta Z_{st}^j +  \int_s^t \delta Z_{sr}^i u(r,X_s)^j dr  
\end{align*}
so that $\tilde{X}^{\natural}$ is remainder. This proves the result.
\end{proof}

%At the present level of understanding, it is not clear whether the mapping $\Z \mapsto X$ is continuous in the rough path topology.

\subsection{Uniqueness} % and rough path continuity}

Finally, we show that solutions to \eqref{MainEq} are unique in law. The proof goes as follows.

Assume we have
$$
dX_t = u(t,X_t)dt + dB_t + \beta_j(X_t) d\Z_t^j, \quad X_0 = x \in \R^d
$$
and define the measure $dQ = \exp \left( - \int_0^T u_j(s,X_s) dB^j_s - \frac12 \int_0^T |u(s,X_s)|^2 ds \right) dP$. Then we know that the process  $dY_t := u(t,X_t)dt + dB_t$ is a $Q$ Brownian motion, 
and if we can show that 
$$
dX_t = dY_t + \beta_j(X_t) d\Z_t^j = V_j(X_t) d\tilde{\Z}(Y)_t^j, \quad X_0 = x \in \R^d
$$
the result follows since the solution $X$ is constructed in a pathwise sense from $\tilde{\Z}(Y)$. 

\begin{theorem}
The solution constructed in Theorem \ref{thm:RSDEexistence} is unique in law.
\end{theorem}

\begin{proof}
By assumption we have 
\begin{align}
\delta X_{st} & =  \int_s^t u(r,X_r) dr + \delta B_{st} + \beta_j(X_s) \delta Z_{st}^j  + \nabla \beta_j(X_s) \beta_i(X_s) \ZZ_{st}^{i,j} \\
 & \qquad + \nabla \beta_j(X_s) e_i \int_s^t \delta B_{sr}^i dZ_r^j + X_{st}^{\natural} .  %\label{eq:expansion}
\end{align}
Define the measure $dQ = \exp \left( - \int_0^T u_j(s,X_s) dB^j_s - \frac12 \int_0^T |u(s,X_s)|^2 ds \right) dP$. Then we know that the process  $dY_t := u(t,X_t)dt + dB_t$ is a $Q$ Brownian motion, and we have
$$
 - \int_0^T u_j(s,X_s) dB^j_s - \frac12 \int_0^T |u(s,X_s)|^2 ds  =  \int_0^T u_j(s,X_s) dY^j_s - \frac12 \int_0^T |u(s,X_s)|^2 ds  .
$$
Similarly as before, we get
\begin{align}
\delta X_{st} & =   \delta Y_{st} + \beta_j(X_s) \delta Z_{st}^j  + \nabla \beta_j(X_s) \beta_i(X_s) \ZZ_{st}^{i,j} \\
 & \qquad + \nabla \beta_j(X_s) e_i \int_s^t \delta Y_{sr}^i dZ_r^j + \tilde{X}_{st}^{\natural} ,  %\label{eq:expansion}
\end{align}
where 
$$
\tilde{X}^{\natural}_{st} := \nabla \beta_j(X_s) e_i \left( \int_s^t \delta B_{sr}^i dZ_r^j -   \int_s^t \delta Y_{sr}^i dZ_r^j \right)
$$ 
which shows that $X  = X^Y$ is the solution to \eqref{eq:PureRoughEq} driven by $Y$ on some set of full measure. Since the construction of the solution of \eqref{eq:PureRoughEq} is purely deterministic, the law of $X^Y$ depends only on the law of $Y$. Consequently,
\begin{align*}
E_P[F(X_{\cdot})] & = E_Q \left[ F(X^Y_{\cdot})  \exp \left( \int_0^T u_j(s,X_s^Y) dY^j_s - \frac12 \int_0^T |u(s,X_s^Y)|^2 ds \right) \right] %\\
\end{align*}
for any bounded and measurable function $F: C([0,T];\R^d) \rightarrow \R$ which shows uniqueness in law.
\end{proof}

%\newpage

\end{document}